\documentclass[10pt]{amsart}
\usepackage{graphicx}
\usepackage{amscd}
\usepackage{amsmath}
\usepackage{amsthm}
\usepackage{amsfonts}
\usepackage{amssymb}
\usepackage{mathrsfs}
\usepackage{bm}
\usepackage{enumitem}
\usepackage{amsrefs}
\usepackage{xcolor}
\usepackage[colorlinks, citecolor=blue, linkcolor=red, pdfstartview=FitB]{hyperref}
\usepackage[latin1]{inputenc}

\newcommand{\bC}{{\mathbb{C}}}
\newcommand{\bD}{{\mathbb{D}}}

\newcommand{\bM}{{\mathbb{M}}}
\newcommand{\bN}{{\mathbb{N}}}

  \newcommand{\A}{{\mathcal{A}}}

  \newcommand{\F}{{\mathcal{F}}}
  \newcommand{\G}{{\mathcal{G}}}

\renewcommand{\S}{{\mathcal{S}}}

  \newcommand{\V}{{\mathcal{V}}}
  \newcommand{\W}{{\mathcal{W}}}

\newcommand{\rC}{\mathrm{C}}

\newcommand{\eps}{\varepsilon}
\renewcommand{\phi}{\varphi}

\newcommand{\upchi}{{\raise.35ex\hbox{$\chi$}}}

\newcommand{\ol}{\overline}

\newcommand{\qand}{\quad\text{and}\quad}

\newcommand{\diag}{\operatorname{diag}}

\newtheorem{lemma}{Lemma}[section]
\newtheorem{theorem}[lemma]{Theorem}

\theoremstyle{definition}

\newtheorem{example}{Example}

\begin{document}
\author{Rapha\"el Clou\^atre}

\address{Department of Mathematics, University of Manitoba, Winnipeg, Manitoba, Canada R3T 2N2}

\email{raphael.clouatre@umanitoba.ca\vspace{-2ex}}
\thanks{The first author was partially supported by an NSERC Discovery Grant.}
\author{Diarra Mbacke}
\email{mbackmdb@myumanitoba.ca\vspace{-2ex}}

\begin{abstract}
An example due to Pisier shows that two commuting, completely polynomially bounded Hilbert space operators may not be simultaneously similar to contractions. Thus, while each operator is individually similar to a contraction, the pair is not \emph{jointly} similar to a pair of commuting contractions. We show that this phenomenon does not occur in finite dimensions. More precisely, we show that a finite family of power bounded commuting matrices is always jointly similar to a family of contractions. In fact, the result can be extended to infinite families satisfying certain uniformity conditions. Our approach is based on a joint spectral decomposition of the underlying space.
\end{abstract}

\title[Joint similarity for commuting matrices]{Joint similarity for commuting families\\ of power bounded matrices}
\maketitle

\section{Introduction}

The classical von Neumann inequality \cite{vN1951} states that for a contractive linear operator $T$ acting on a Hilbert space, we always have that
\[
\|f(T)\|\leq \sup_{z\in \bD}|f(z)|
\]
for every polynomial $f$, where $\bD$ denotes the open unit disc in the complex plane. This observation lies at the base of the fruitful connection between complex function theory and operator theory. It also provides motivation for one of Halmos' famous ten problems \cite{halmos1970}, essentially asking to characterize the class of \emph{polynomially bounded} operators, that is those Hilbert space operators for which von Neumann's inequality holds up to a multiplicative constant. It is readily seen that being similar to a contraction is a sufficient condition for an operator to be polynomially bounded, and Halmos asked whether this condition was in fact necessary. If the condition is weakened to the operator merely having uniformly bounded powers, then this was shown not to be the case by Foguel \cite{foguel1964}.

A key insight into Halmos' question was provided by Paulsen \cite{paulsen1984}, who showed that an operator is similar to a contraction if and only if it is \emph{completely} polynomially bounded, in the sense that it satisfies von Neumann's inequality up to a multiplicative constant for arbitrary matrix-valued polynomials. Such a characterization turned out to be very fruitful, and led to the solution of Halmos' problem by Pisier \cite{pisier1997}. Therein, an example is exhibited of a polynomially bounded operator which is not completely polynomially bounded. A somewhat streamlined treatment appears in \cite{davidsonpaulsen1997}. We refer the interested reader to \cite[Chapter 10]{paulsen2002}  or \cite[Chapter 28]{pisier2003} for a detailed account of this problem and its solution.

In this paper we explore the aforementioned problem in a multivariate context. This is motivated by \cite[Theorem 1]{pisier1998}, where it is shown that there exist two commuting bounded linear operators $S$ and $T$ on Hilbert space which are individually similar to contractions (equivalently, $S$ and $T$ are both completely polynomially bounded), yet there is no single invertible operator $Y$ with the property that $YSY^{-1}$ and $YTY^{-1}$ are both contractions. In fact, the product $ST$ is not even polynomially bounded, thus extending the result of \cite{petrovic1997}. This multivariate twist has received a fair amount of attention in various special cases \cite{FG2002},\cite{CJ2009},\cite{constantin2010}. Decisive results were obtained in \cite{popescu2014sim} for joint similarity to strict contractions.

We mention that this joint similarity problem can be recast in the setting of operator algebras by reformulating it in terms of a certain notion of length. In that language, \cite[Theorem 1]{pisier1998} says that the maximal tensor product of the familiar disc algebra with itself has infinite length. In contrast, it is shown in \cite{pisier1998} that the maximal tensor product of a nuclear $\rC^*$-algebra with any unital operator algebra turns out to have finite length; see also \cite{pisier2007} for related results.

The focus of our work here is more modest. We investigate the question of simultaneous similarity to contractions for commuting matrices. In this finite-dimensional setting, we show that the phenomenon exhibited in \cite[Theorem 1]{pisier1998} does not occur. 

We now describe the organization of the paper. Section \ref{S:prelim} gathers the necessary background and some preliminary tools that are used throughout. In Section \ref{S:decomp}, we establish the following spectral decomposition for commuting families of matrices with spectra in $\ol{\bD}$  (Theorem \ref{T:decompfamily}). This is  our main technical tool and may be of independent interest. For a matrix $T$, we denote by $\Delta(T)$ the subset of the spectrum consisting of those eigenvalues that appear in a block of size at least $2$ in the Jordan canonical form of $T$.
\begin{theorem}\label{T:maindecomp}
Let $\A$ be a commuting family of $n\times n$ matrices with spectra in $\ol{\bD}$. Then, there are finitely many non-zero subspaces $V_1,\ldots,V_{s}\subset \bC^n$ with the following properties:
\begin{enumerate}

\item[\rm{(a)}] for each $1\leq i\leq s$, the subspace $V_i$ is invariant for $\A$;

\item[\rm{(b)}] we have $V_i\cap \left( \sum_{j\neq i}V_j\right)=\{0\}$ for every $1\leq i\leq s$ and $\bC^n=\sum_{i=1}^s V_i$;

\item[\rm{(c)}]  for each $1\leq i\leq s$ and each $T\in \A$, either $\sigma(T|_{V_i})\subset \Delta(T)$ or there is $z\in \ol{\bD}$ such that $T|_{V_i}=z I_{V_i}$.

\end{enumerate}
\end{theorem}

In Section \ref{S:jointsim}, this decomposition is used to establish our main results. The first one deals with finite families of commuting matrices (Theorem \ref{T:jointsimfinitepoly}).

\begin{theorem}\label{T:mainjointsimfinitepoly}
Let $T_1,\ldots,T_m$ be commuting power bounded matrices. Then, there exists an invertible matrix $Y$ with the property that $YT_k Y^{-1}$ is a contraction for every $1\leq k\leq m$. 
\end{theorem}

In fact, we obtain a precise estimate on the size of the similarity $Y$ in the previous theorem. This information is then leveraged to extend the result to infinite families satisfying certain uniformity conditions (Theorem \ref{T:jointsiminfinite}). Roughly speaking, the family $\A$ should be uniformly power bounded, and there should be a uniform bound on the size of the similarity needed to put any given element of $\A$ in Jordan canonical form.

\begin{theorem}\label{T:mainjointsiminfinite}
Let $\A$ be a uniformly power bounded commuting family of matrices  with the uniform Jordan property. Then, there exists an invertible matrix $Y$ with the property that $YT Y^{-1}$ is a contraction for every $T\in \A$. 
\end{theorem}

\textbf{Acknowledgements.} This project was initiated while the second author was conducting summer research under the supervision of the first author. The second author wishes to acknowledge the financial support of the Faculty of Science at the University of Manitoba through an Undergraduate Student Research Award.

\section{Background and preliminary results}\label{S:prelim}

 \subsection{Basic facts from linear algebra}
 
Given a positive integer $n$, we denote by $\bM_n$ the space of $n\times n$ complex matrices. We fix once and for all an orthonormal basis $\{e_1,\ldots,e_n\}$ of the Hilbert space $\bC^n$ and identify $\bM_n$ with the space of linear operators on $\bC^n$ in the usual fashion. Recall that if $\A\subset \bM_n$ is a subset and $V\subset \bC^n$ is a subspace, then $V$ is said to be \emph{invariant for $\A$} if $TV\subset V$ for every $T\in \A$. We write $\A'$ for the \emph{commutant} of $\A$, that is the collection of all matrices in $\bM_n$ commuting with every element of $\A$. The \emph{spectrum} of a matrix $T$ is the subset $\sigma(T)\subset \bC$ consisting of its eigenvalues. We will require the following basic rigidity property for matrices commuting with a given block diagonal matrix.

\begin{lemma}\label{L:commutant}
Let $A_1\in \bM_{n_1}, A_2\in \bM_{n_2},\ldots ,A_d \in \bM_{n_d}$ be matrices with pairwise disjoint spectra and let $T=A_1\oplus \ldots \oplus A_d$.  Let $S$ be a matrix that commutes with $T$. Then, $S=B_1\oplus B_2\oplus \ldots \oplus B_d$ where $B_i\in \bM_{n_i}$ commutes with $A_i$ for every $1\leq i\leq d$.
\end{lemma}
\begin{proof}
Decomposing the matrix $S$ according to the block decomposition of $T$, we may write $S=[B_{ij}]_{i,j=1}^d$ for some rectangular matrices $B_{ij}\in \bM_{n_i\times n_j}$ such that $A_i B_{ij}=B_{ij} A_j$ for every $1\leq i,j\leq d$. By Sylvester's theorem \cite[Theorem 2.4.4.1]{HJ2013}, the assumption that the spectra of $A_i$ and $A_j$ are disjoint  shows that $B_{ij}=0$ whenever $i\neq j$. The proof is finished by defining $B_i=B_{ii}$ for every $1\leq i\leq d$.
\end{proof}

Some of our arguments will make crucial use of the Jordan structure of matrices, which we now recall. If $r$ is a positive integer and $\lambda$ is a complex number, we let $J_r(\lambda)$ denote the \emph{Jordan block} of size $r$ with eigenvalue $\lambda$, that is
\[
J_r(\lambda)=\begin{bmatrix}

\lambda & 1 & & & & &\\
 & \lambda & 1 & & &\\
 & & \lambda & 1& &  \\
 & & & \ddots & \ddots\\
 & & & & \lambda & 1\\
 & & & & & \lambda
\end{bmatrix}\in \bM_r
\]
where the unspecified entries are zero. Given an arbitrary matrix $T\in \bM_n$,  there is an invertible matrix $X\in \bM_n$ such that $XTX^{-1}$ is in \emph{Jordan canonical form}, in the sense that there are positive integers $r_1,\ldots,r_d$ along with complex numbers $\lambda_1, \ldots,\lambda_d\in \sigma(T)$ such that
\[
XTX^{-1}=J_{r_1}(\lambda_1)\oplus \ldots \oplus J_{r_d}(\lambda_d).
\]
This form is unique up to a permutation of the Jordan blocks \cite[Theorem 3.1.11]{HJ2013}. We let $\Delta(T)\subset \sigma(T)$ denote the subset consisting of those $\lambda\in \sigma(T)$ for which the Jordan canonical form of $T$ contains a block of the form $J_r(\lambda)$ for some $r\geq 2$. Equivalently, an eigenvalue $\lambda$ lies in $\Delta(T)$ if and only if 
\[
\ker(T-\lambda I)\neq \ker (T-\lambda I)^2.
\]
An elementary property that we will require is the following.

\begin{lemma}\label{L:Deltarest}
Let $T\in \bM_n$ be a matrix. Assume that there are subspaces $V,W\subset \bC^n$ which are invariant for $T$ such that $V\cap W=\{0\}$ and $\bC^n=V+W$. Then, 
\[
\Delta(T|_V)\cup \Delta(T|_W)=\Delta(T).
\]
\end{lemma}
\begin{proof}
Define $X:\bC^n\to V\oplus W$ as $X(v+w)=(v,w)$ for every $v\in V,w\in W.$ By assumption, we see that $X$ is a well-defined invertible linear  operator and that $XTX^{-1}=T|_V\oplus T|_W$. Denote by $J, J_V$ and $J_W$ the Jordan canonical forms of $T, T|_V$ and $T|_W$ respectively. By uniqueness of the Jordan canonical form  we see that $J$ is unitarily equivalent to $J_V\oplus J_W$, and thus
\[
\Delta(T|_V)\cup \Delta(T|_W)=\Delta(T).
\]
\end{proof}

We define the norm of a matrix $T\in \bM_n$ to be that of the associated linear operator on the finite-dimensional Hilbert space $\bC^n$, so that
\[
\|T\|=\max\{\|Tv\|:v\in \bC^n, \|v\|=1\}.
\]
We say that $T$ is a \emph{contraction} if $\|T\|\leq 1$. If $T=[t_{ij}]_{i,j=1}^n$, then
\[
\max\{|t_{ij}|:1\leq i,j\leq n\}\leq \|T\|\leq n^2 \max\{|t_{ij}|:1\leq i,j\leq n\}.
\]
Moreover, if $T=\diag(t_1,\ldots,t_n)$, then 
\[
\|T\|=\max\{|t_j|:1\leq j\leq n\}.
\]
More generally, if there are matrices $T_j\in \bM_{n_j},1\leq j\leq d$ such that 
\[
T=T_1\oplus T_2\oplus \ldots\oplus T_d,
\]
then 
\[
\|T\|=\max\{\|T_j\|:1\leq j\leq d\}.
\]
We say that $T$ is \emph{power bounded} if there is a constant $K>0$ with the property that $\|T^p\|\leq K$ for every $p\in \bN$. To emphasize the constant $K$, we sometimes also say that $T$ is \emph{power bounded with constant $K$}. A collection of matrices $\A$ will be said to be \emph{uniformly power bounded} if there is $K>0$ such that every $T\in \A$ is power bounded with constant $K$.

Our next task is to establish  a useful fact about the spectrum of a power bounded matrix. To facilitate this, we introduce the following notation:
\[
\delta(T)=\min_{\lambda\in \Delta(T)}\|(T-\lambda I)|_{\ker(T-\lambda I)^2}\|.
\]
We note that $\delta(T)>0$ by definition of $\Delta(T)$. Throughout the paper we denote by $\bD_r\subset \bC$ the open disc of radius $r>0$ centred at the origin. The topological closure of $\bD_r$ is denoted by $\ol{\bD_r}$. When $r=1$, we simply write $\bD$ instead of $\bD_1$. For convenience, we adopt the convention that the maximum of the empty set is simply $0$.

\begin{lemma}\label{L:powerbounded}
Let $T\in \bM_n$ be a matrix which is power bounded with constant $K>0$.
Then, we have that $\sigma(T)\subset \ol{\bD}$ and
\[
\max_{\lambda\in \Delta(T)}\sup_{p\in \bN }\{p|\lambda|^{p-1}\}\leq \frac{K}{\delta(T)}.
\]
In particular, we have $\Delta(T)\subset \bD$.
\end{lemma}
\begin{proof}
By assumption, we have $\|T^p\|\leq K$ for every $p\in \bN$. Let $\lambda\in \sigma(T)$ and consider $V=\ker(T-\lambda I)$. Then, $V$ is invariant for $T$ and $T^p|_V=\lambda^p I$ whence 
\[
|\lambda|^p=\|T^p|_V\|\leq \|T^p\|\leq K
\]
for every $p\in \bN$. We infer $|\lambda|\leq 1$. This shows that $\sigma(T)\subset \ol{\bD}$. 

Next, assume that $\lambda\in \Delta(T)$ and let $W=\ker(T-\lambda I)^2$. Then, $W$ is invariant for $T$ and we set $R=T|_W$. Note that $(R-\lambda I)^2=0$ whence 
\[
(R-\lambda I)W\subset \ker (R-\lambda I).
\]
According to the decomposition
\[
W=(R-\lambda I)W\oplus ((R-\lambda I)W)^\perp
\]
we may write
\[
R=\begin{bmatrix}
\lambda & S\\
0 & \lambda
\end{bmatrix}
\]
for some linear operator $S$ satisfying 
\[
\|S\|=\|R-\lambda I\|=\|(T-\lambda I)|_{\ker (T-\lambda I)^2}\|\geq \delta(T).
\]
For each $p\in \bN$, it is easily verified that the $(1,2)$-entry of $R^p$ is $p \lambda^{p-1}S$, and therefore we find
\[
p|\lambda|^{p-1}\|S\|\leq \|R^p\|\leq \|T^p\|\leq K
\]
and
\[
p|\lambda|^{p-1}\leq\frac{K}{\|S\|}\leq \frac{K}{\delta(T)}.
\]
In particular, $\lim_{p\to\infty}|\lambda|^p=0$ for every $\lambda\in \Delta(T)$, which implies that $\Delta(T)\subset \bD$.
\end{proof}

 \subsection{Matrices similar to contractions}
 Our main focus in the paper will be the similarity of certain matrices to contractions. We record here a particularly simple case of a classical theorem of Rota \cite{rota1960} (a multivariate generalization can be found in \cite{popescu2014sim}). We provide an elementary proof.
 
\begin{lemma}\label{L:opendisc}
Let $T_1,\ldots,T_m\in \bM_n$ be commuting matrices. Assume that there are constants $K>1$ and $0<r<1$ such that for every $1\leq k\leq m$, we have $\sigma(T_k)\subset \ol{\bD_r}$ and $\|T_k\|\leq K$. Then, there is an invertible matrix $Y\in \bM_n$ such that $Y T_k Y^{-1}$ is a contraction for every $1\leq k\leq m$ and with the property that 
\[
\|Y\|= \|Y^{-1}\|\leq \left( \frac{n^2 K}{1-r}\right)^{\frac{n-1}{2}}.
\]
\end{lemma}
\begin{proof}
Since the family $\{T_1,\ldots,T_m\}$ is commuting,  by virtue of \cite[Theorem 2.4.8.7]{HJ2013}, there is a unitary matrix $U\in \bM_n$ with the property that $UT_k U^{-1}$ is upper triangular for every $1\leq k\leq m$. Thus, for each $1\leq k\leq m$ there are complex numbers $t^{(k)}_{ij} \in \bC, 1\leq i<j\leq n$ and $\tau^{(k)}_j\in \bC ,1 \leq j\leq n$ such that
\[
U T_k U^{-1}=\begin{bmatrix}
\tau^{(k)}_1 & t^{(k)}_{12} & t^{(k)}_{13} & t^{(k)}_{14}&  \cdots & t^{(k)}_{1n}\\
 & \tau^{(k)}_2 & t^{(k)}_{23} & t^{(k)}_{24} & \cdots & t^{(k)}_{2n}\\
 & &  \tau^{(k)}_3 & t^{(k)}_{34} & \cdots & t^{(k)}_{3n}\\
 &  &   &  \ddots&  \ddots& \vdots \\
  &  &  &  &  \tau^{(k)}_{n-1}  & t_{n-1,n}^{(k)}\\
   &  &  &  &  & \tau^{(k)}_n
\end{bmatrix}
\] 
where the unspecified entries are zero. By assumption on $T_k$, for every $1\leq k\leq m$ we see that $|\tau^{(k)}_j|\leq r$ for every $1\leq j\leq n$ and $|t_{ij}^{(k)}|\leq K$ for every $1\leq i<j\leq n$. Let
\[
\eps=\frac{1-r}{n^2K}
\]
and define
\[
X=\diag(1,\eps^{-1}, \eps^{-2},\ldots,\eps^{-(n-1)})\in \bM_n.
\]
We see that $\|X^{-1}\|\leq 1$ and 
\[
\|X\|= \eps^{-(n-1)}=\left( \frac{n^2 K}{1-r}\right)^{n-1}.
\]
For convenience, for each $1\leq k \leq m$ we set
\[
R_k=\begin{bmatrix}
0 & \eps t^{(k)}_{12} & \eps^2 t^{(k)}_{13} & \eps^ 3 t^{(k)}_{14}&  \cdots & \eps^{n-1}t^{(k)}_{1n}\\
 & 0& \eps t^{(k)}_{23} & \eps^2  t^{(k)}_{24} & \cdots & \eps^{n-2} t^{(k)}_{2n}\\
 & &  0 & \eps t^{(k)}_{34} & \cdots & \eps^{n-3} t^{(k)}_{3n}\\
 & &   & \ddots &  \ddots & \vdots\\
   &  &  &  & 0  & \eps t_{n-1,n}^{(k)}\\
 &  &  &  & &0
\end{bmatrix}
\]
where the unspecified entries are zero. A routine calculation now yields
\[
XU T_k U^{-1}X^{-1}=\diag(\tau^{(k)}_1,\tau^{(k)}_2,\ldots,\tau^{(k)}_n)+R_k
\] 
so we infer that
\begin{align*}
\|XU T_k U^{-1}X^{-1}\|&\leq \|\diag(\tau^{(k)}_1,\tau^{(k)}_2,\ldots,\tau^{(k)}_n)\|+\|R_k\|\\
&\leq \max\{|\tau^{(k)}_j|:1\leq j\leq n\}+n^2 \max\{ \eps^{j-i} |t^{(k)}_{ij}|:1\leq i<j\leq n\}\\
&\leq r+n^2 K\eps=1
\end{align*}
for every $1\leq k\leq m$. Therefore, the proof is complete upon setting 
\[
Y=\left(\frac{\|X^{-1}\|}{\|X\|}\right)^{1/2}XU.
\]
\end{proof}

\section{The spectral decomposition}\label{S:decomp}
This section contains the brunt of the technical work underlying our main arguments. Our goal is to show that the space $\bC^n$ can be decomposed as a direct sum in a manner that is compatible with the spectral properties of a given commuting family of matrices. Such a decomposition will then allow us to leverage Lemmas  \ref{L:powerbounded} and \ref{L:opendisc}.

The first step in achieving the desired spectral decomposition is the following.
Given a vector space $V$, we denote by $I_V$ the identity operator on it.

\begin{lemma}\label{L:decomppb=1}
Let $T\in \bM_n$ be a matrix with spectrum in $\ol{\bD}$. Then, there are finitely many non-zero subspaces $V_1,\ldots,V_{s}\subset \bC^n$ with the following properties:
\begin{enumerate}

\item[\rm{(a)}] for each $1\leq i\leq s$, the subspace $V_i$ is invariant for $\{T\}'$;

\item[\rm{(b)}] we have $V_i\cap \left( \sum_{j\neq i}V_j\right)=\{0\}$ for every $1\leq i\leq s$ and $\bC^n=\sum_{i=1}^s V_i$;
\item[\rm{(c)}]  for each $1\leq i\leq s$, either $\sigma(T|_{V_i})= \Delta(T)$ or there is $z\in \ol{\bD}$ such that $T|_{V_i}=z I_{V_i}$.
\end{enumerate}
\end{lemma}
\begin{proof}
Choose an invertible matrix $X\in \bM_n$ such that $J=X^{-1}TX$, where $J$ denotes the Jordan canonical form of $T$. 

Assume first that $\Delta(T)$ is empty.  By definition of $\Delta(T)$, we see that there are finitely many non-zero subspaces $W_1,\ldots,W_s\subset \bC^n$ that are invariant for $J$ and such that $\bC^n=\oplus_{i=1}^s W_i$, along with distinct complex numbers $z_1,\ldots,z_s\in \ol{\bD}$ such that  $J|_{W_i}=z_i I_{W_i}$ for every $1\leq i\leq s$. It is a consequence of Lemma \ref{L:commutant} that the subspaces $W_1,\ldots,W_s$ are in fact invariant for $\{J\}'$. Let $V_i=XW_i$ for every $1\leq i\leq s$. Then, $V_1,\ldots,V_s$ are invariant for $\{T\}'$ and have all the desired properties, so the proof is complete in this case.

If $\sigma(T)= \Delta(T)$ then the desired conclusion trivially holds with $V_1=\bC^n$.  The remaining case is that when $\Delta(T)$ and $\sigma(T)\setminus \Delta(T)$ are both non-empty. In this case, there are non-zero invariant subspaces $W_1,\ldots,W_s\subset \bC^n$ for $J$ such that  $\sigma(J|_{W_1})=\Delta(T)$, while there are distinct complex numbers $z_2,\ldots,z_s\in \sigma(T)\setminus \Delta(T)$ such that $J|_{W_i}=z_i I_{W_i}$ for every $2\leq i\leq s$. As above, it follows from Lemma \ref{L:commutant} that the subspaces $W_1,\ldots,W_s$ are in fact invariant for $\{J\}'$. For each $1\leq i\leq s$, we let $V_i=XW_i$. Then, the subspaces $V_1,\ldots,V_s$ have all the desired properties and the conclusion is established in this case as well.
\end{proof}

We now arrive at our main technical tool, which extends the spectral decomposition of the previous lemma to arbitrary commuting families.

\begin{theorem}\label{T:decompfamily}
Let $\A\subset \bM_n$ be a commuting family of matrices with spectra in $\ol{\bD}$. Then, there are finitely many non-zero subspaces $V_1,\ldots,V_{s}\subset \bC^n$ with the following properties:
\begin{enumerate}

\item[\rm{(a)}] for each $1\leq i\leq s$, the subspace $V_i$ is invariant for $\A$;

\item[\rm{(b)}] we have $V_i\cap \left( \sum_{j\neq i}V_j\right)=\{0\}$ for every $1\leq i\leq s$ and $\bC^n=\sum_{i=1}^s V_i$;

\item[\rm{(c)}]  for each $1\leq i\leq s$ and each $T\in \A$, either $\sigma(T|_{V_i})\subset \Delta(T)$ or there is $z\in \ol{\bD}$ such that $T|_{V_i}=z I_{V_i}$.

\end{enumerate}
\end{theorem}
\begin{proof}
Let $\S$ be the collection of pairs $(\F,\V)$ where $\F\subset \A$  is a subset and $\V$ is a collection of non-zero subspaces of $\bC^n$ with the following properties:
\begin{itemize}

\item every $V\in \V$ is invariant for $\A$;

\item we have $V\cap \left(\sum_{W\in \V, W\neq V} W\right)=\{0\}$ for every $V\in \V$ and $\bC^n=\sum_{V\in \V}V$;

\item  for each $1\leq i\leq s$ and each $T\in \F$, either $\sigma(T|_{V_i})\subset \Delta(T)$ or there is $z\in \ol{\bD}$ such that $T|_{V_i}=z I_{V_i}$.
\end{itemize}
Notice that $\S$ is non-empty by Lemma \ref{L:decomppb=1}. Moreover, the second property above forces the cardinality of $\V$ to be at most $n$. In particular, we may choose a pair $(\F,\V)\in \S$ with the property that $\V$ has maximal cardinality. Write $\V=\{V_1,\ldots,V_s\}$. We claim that the subspaces $V_1,\ldots,V_s$ have the desired properties. We see that (a) and (b) are automatically satisfied, so it suffices to establish (c). Let $T\in \A$. We must verify that for every $1\leq i\leq s$, either $\sigma(T|_{V_i})\subset \Delta(T)$ or there is $z\in \ol{\bD}$ such that $T|_{V_i}=z I_{V_i}$. 

Assume otherwise, so that there is a non-empty subset $\Lambda$ of indices $1\leq i\leq s$ with the property that the restriction $T|_{V_i}$ is not of the form $zI_{V_i}$ for some $z\in \ol{\bD}$, and $\sigma(T|_{V_i})$ is not contained in $\Delta(T)$.  Fix $i\in \Lambda$. We may apply Lemma \ref{L:decomppb=1} to find finitely many non-zero subspaces $W^{(i)}_1,\ldots,W^{(i)}_{q_i}\subset V_i$ with the following properties:
\begin{itemize}

\item  $W^{(i)}_j$ is invariant for $\{T|_{V_i}\}'$ for each $1\leq j\leq q_i$;

\item we have $W^{(i)}_{j_0}\cap \left( \sum_{j\neq j_0}W^{(i)}_j\right)=\{0\}$ for every $1\leq j_0\leq q_i$ and $V_i=\sum_{j=1}^{q_i} W^{(i)}_j$;
\item  for each $1\leq j\leq q_i$,  either $\sigma(T|_{W^{(i)}_j})= \Delta(T|_{V_i})$  or there is $z\in \ol{\bD}$ such that $T|_{W^{(i)}_j}=z I_{W^{(i)}_j}$.
\end{itemize}
We have that  $\Delta(T|_{V_i})\subset \Delta(T)$ by Lemma \ref{L:Deltarest}, so  the third property above forces $q_i\geq 2$. Observe now that $A|_{V_i}\in \{T|_{V_i}\}'$ if $A\in \A$. Therefore, for each $1\leq j\leq q_i$ the subspace $W_j^{(i)}$ is invariant for the family $\A$. In particular, this implies that $\sigma(A|_{W^{(i)}_j})\subset \sigma(A|_{V_i})$ for every $1\leq j\leq q_i$ and every $A\in \A$.  If we let $\G=\F\cup \{T\}$ and
\[
\W=\{V_i:1\leq i\leq s, i\notin \Lambda\}\cup \{W_j^{(i)}:i\in \Lambda, 1\leq j\leq q_i\}
\]
then we see $(\G,\W)\in \S$, which contradicts the maximality property of $(\F,\V)$ as $q_i\geq 2$ for every $i\in \Lambda$.
\end{proof}

\section{Joint similarity}\label{S:jointsim}

In this section, we prove our main results based on the spectral decomposition obtained in Theorem \ref{T:decompfamily}. We first deal with the case of finitely many matrices, starting from the following observation. Recall that we adopt the convention that the maximum of the empty set  is $0$.

\begin{lemma}\label{L:jointsimVi}
Let $\F\subset \bM_n$ be a finite subset of commuting matrices. Assume that there is a subset $\F'\subset \F$ with the following properties:
\begin{enumerate}

\item[\rm{(a)}] if $T\in \F'$, then there is $z\in \ol{\bD}$ such that $T=z I$;

\item[\rm{(b)}] if $T\in \F\setminus \F'$, then $\sigma(T)\subset \bD$.
\end{enumerate}
Then, there exists an invertible matrix $Y\in \bM_n$ with the property that $YT Y^{-1}$ is a contraction for every $T\in \F$. Moreover, we have
\[
\|Y\|=\|Y^{-1}\|\leq  \left( \frac{n^2 K}{1-r}\right)^{\frac{n-1}{2}}
\]
where
\[
K=\max_{T\in \F}\|T\| \qand r=\max_{T\in \F\setminus \F'}\max_{\lambda\in \sigma(T)} |\lambda|.
\]
\end{lemma}
\begin{proof}
We proceed by induction on the cardinality of $\F$. Assume first that $\F$ has one element, say $T$. If $T\in \F'$, there is nothing to prove. Assume thus that $\sigma(T)\subset \bD$. We may apply Lemma \ref{L:opendisc} to find an invertible matrix $Y\in \bM_n$ such that $YTY^{-1}$ is a contraction and
\[
\|Y\|=\|Y^{-1}\|\leq \left( \frac{n^2 K}{1-r}\right)^{\frac{n-1}{2}}.
\]
Assume that the conclusion holds whenever $\F$ has $m$ elements.  We claim that the conclusion holds when $\F$ has $m+1$ elements as well. To see this, note first that if $\F'=\varnothing$, then an application of Lemma \ref{L:opendisc} yields the existence of an invertible matrix $Y$ with the property that $Y T Y^{-1}$ is a contraction for every $T\in \F$ and
\[
\|Y\|= \|Y^{-1}\|\leq \left( \frac{n^2 K}{1-r}\right)^{\frac{n-1}{2}}.
\]
This case did not require the induction hypothesis. In the alternative situation, there is $T'\in \F'$ and $z\in \ol{\bD}$ such that $T'=zI $.  By the induction hypothesis, there is an invertible matrix $Y$ with the property that $YTY^{-1}$ is a contraction for every $T\in\F\setminus \{T'\}$ and
\[
\|Y\|=\|Y^{-1}\|\leq   \left( \frac{n^2 K}{1-r}\right)^{\frac{n-1}{2}}.
\]
Trivially, we see that
\[
YT' Y^{-1}=Y(z I)Y^{-1}=z I
\]
is also a contraction. In either case, we have found an invertible matrix $Y$ with the property that $Y TY^{-1}$ is a contraction for every $T\in \F$ and
\[
\|Y\|=\|Y\|\leq  \left( \frac{n^2 K}{1-r}\right)^{\frac{n-1}{2}},
\]
so the proof is complete by induction.
\end{proof}

We can now prove one of our main results.

\begin{theorem}\label{T:jointsimfinitepoly}
Let $\A \subset \bM_n$ be a commuting family of power bounded matrices  and let $\F\subset \A$ be a finite subset. Then, there exists an invertible matrix $Y\in \bM_n$ with the property that $YT Y^{-1}$ is a contraction for every $T\in \F$. Moreover, we have
\[
\|Y\|=\|Y^{-1}\|\leq \alpha \left( \frac{n^2 K}{1-r}\right)^{\frac{n-1}{2}}
\]
where $\alpha\geq 1$ is a constant depending only on $\A$,
\[
K=\max_{T\in \F}\|T\| \qand r=\max_{T\in \F}\max_{\lambda\in \Delta(T)}|\lambda|.
\]

\end{theorem}
\begin{proof}
Invoking Lemma \ref{L:powerbounded}, we see that every matrix in $\A$ has spectrum contained in $\ol{\bD}$. We may thus apply Theorem \ref{T:decompfamily} to find finitely many non-zero subspaces $V_1,\ldots,V_{s}\subset \bC^n$ with the following properties:
\begin{itemize}

\item  for each $1\leq i\leq s$, the subspace $V_i$ is invariant for $\A$;

\item we have $V_i\cap \left( \sum_{j\neq i}V_j\right)=\{0\}$ for every $1\leq i\leq s$ and $\bC^n=\sum_{i=1}^s V_i$;

\item for each $1\leq i\leq s$ and each $T\in \A$, either $\sigma(T|_{V_i})\subset \Delta(T)$  or there is $z\in \ol{\bD}$ such that $T|_{V_i}=z I_{V_i}$.

\end{itemize}
Define a linear map $X :\sum_{i=1}^{s}V_i\to \oplus_{i=1}^{s} V_i$ as 
\[
X\left( \sum_{i=1}^{s} v_i\right)=(v_1,v_2,\ldots,v_{s})
\]
for each $v_1\in V_1,\ldots, v_{s}\in V_{s}$. It is readily seen that $X$ is  a well-defined invertible operator. Set
\[
\alpha=\max\{\|X\|,\|X^{-1}\|\}.
\]
For each $T\in \A$, we note that
\[
X T X^{-1}=T|_{V_1}\oplus T|_{V_2}\oplus \ldots \oplus T|_{V_{s}}.
\]
Moreover, another application of Lemma \ref{L:powerbounded} reveals that $\Delta(T)\subset \bD$ for every $T\in \A$. Now, for each $1\leq i\leq s$ we observe that $\{T|_{V_i}:T\in \F\}$ is a finite set of commuting matrices.  Thus, for $1\leq i\leq s$ we may apply Lemma \ref{L:jointsimVi} to find an invertible matrix $Z_i$ with the property that $Z_i (T|_{V_i})Z_i^{-1}$ is a contraction for every $T\in \F$ and
\[
\|Z_i\|=\|Z_i^{-1}\|\leq  \left( \frac{n^2 K}{1-r}\right)^{\frac{n-1}{2}}.
\]
Define 
\[
Z=(Z_1\oplus \ldots\oplus Z_s)X.
\]
We notice that $ZT Z^{-1}$ is a contraction for every $T\in \F$, so we are done upon setting
\[
Y=\left( \frac{\|Z^{-1}\|}{\|Z\|}\right)^{1/2}Z.
\]
\end{proof}

We remark that the commutativity assumption in Theorem \ref{T:jointsimfinitepoly} cannot be dispensed with.

\begin{example}\label{E:nc}
Consider
\[
T=\begin{bmatrix}
0 & 2\\
0 & 0
\end{bmatrix}.
\]
Then, $T^2=0$ so $T$ and $T^*$ are power bounded with constant $2$. Note that
\[
T^*T=\begin{bmatrix}
0 & 0\\
0 & 4
\end{bmatrix}, \quad
TT^*=\begin{bmatrix}
4 & 0\\
0 & 0
\end{bmatrix}
\]
hence $T$ and $T^*$ do not commute. We claim that there is no invertible matrix $X\in \bM_2$ such that $XTX^{-1}$ and $XT^*X^{-1}$ are contractions. For using that
\[
TT^*=X^{-1}(XTX^{-1})(XT^*X^{-1})X,
\]
we would then find
\begin{align*}
4^p&=\|(TT^*)^p\|\\
&\leq \|X^{-1}\| \|X\| \|(XTX^{-1}) (XT^*X^{-1})\|^p\\
&\leq \|X^{-1}\| \|X\| \|XTX^{-1}\|^p \| XT^*X^{-1}\|^p\\
&\leq \|X^{-1}\|\|X\|
\end{align*}
for every $p\in \bN$, which is absurd.
\qed
\end{example}

Next, we wish to take advantage of the precise estimate from Theorem \ref{T:jointsimfinitepoly} to extend the statement to commuting families of arbitrary cardinality.  The following elementary example shows that some care must be taken in trying to achieve this goal.

\begin{example}\label{E:infinite}
For each $k\in \bN$, let
\[
T_k=\begin{bmatrix}
0 & k\\
0 & 0
\end{bmatrix}\in \bM_2.
\]
We note that $T_k^2=0$ and thus $T_k$ is power bounded for every $k\in \bN$. Moreover, the family $\{T_k:k\in \bN\}\subset \bM_2$ is clearly commuting. Nevertheless, there is no invertible matrix $X\in \bM_2$ such that $XT_k X^{-1}$ is a contraction for every $k\in \bN$. Indeed, this would force
\[
\|T_k\|\leq \|X\| \|X^{-1}\| \|XT_k X^{-1}\|\leq \|X\| \|X^{-1}\|
\]
for every $k\in \bN$, which is absurd since $T_k$ is easily seen to have norm $k$.
\qed
\end{example}

To circumvent this problem, we will consider families of matrices that are uniformly power bounded. In addition, we will require the families to enjoy another type of uniformity. A subset $\A\subset \bM_n$ will be said to have the \emph{uniform Jordan property} if 
\[
\inf_{T\in \A}\min_{\lambda\in \Delta(T)}\|(T-\lambda I)|_{\ker (T-\lambda I)^2}\|>0.
\]
Using the notation introduced before Lemma \ref{L:powerbounded}, we see that $\A$ has the uniform Jordan property if and only if $\inf_{T\in \A}\delta(T)>0$. 

Before proceeding, we wish to exhibit a condition that is sufficient for a collection $\A\subset \bM_n$ to have the uniform Jordan property. Let $T\in \A$ and let $\lambda\in \Delta(T)$. Let $X\in \bM_n$ be an invertible matrix such that $XTX^{-1}=J$, where $J$ is the Jordan canonical form of $T$.  Let $Y=X|_{\ker(T-\lambda I)^2}$. Then, $Y$ is an invertible operator from $\ker (T-\lambda I)^2$ onto $\ker (J-\lambda I)^2$, and $Y^{-1}=X^{-1}|_{\ker (J-\lambda I)^2}$. Hence, we find $\|Y\|\leq \|X\|$ and $\|Y^{-1}\|\leq \|X^{-1}\|$, while
\[
Y (T-\lambda I)|_{\ker(T-\lambda I)^2}Y^{-1}=(J-\lambda I)|_{\ker(J-\lambda I)^2}.
\]
On the other hand, it is easily verified that 
\[
\|(J-\lambda I)|_{\ker(J-\lambda I)^2}\|=1
\]
and therefore
\[
\|(T-\lambda I)|_{\ker(T-\lambda I)^2}\|\geq \frac{1}{\|Y\| \|Y^{-1}\|}\geq \frac{1}{\|X\| \|X^{-1}\|}.
\]
We conclude that $\A$ has the uniform Jordan property if for every $T\in \A$ there is an invertible matrix $X_T\in \bM_n$ with the property that $X_T T X_T^{-1}$ is in Jordan canonical form and such that the quantity 
\[
\sup_{T\in \A} \{\|X_T\|\|X_T^{-1}\|\}
\]
is finite. This explains our choice of terminology.

Next, we unravel the spectral information contained in the uniform Jordan property that is relevant for our purposes.

\begin{lemma}\label{L:unifDelta}
Let $\A\subset \bM_n$ be a uniformly power bounded family of matrices with the uniform Jordan property. Then,
\[
\sup_{T\in \A}\max_{\lambda\in \Delta(T)}|\lambda|<1.
\]
\end{lemma}
\begin{proof}
Choose $K>0$ such that every $T\in \A$ is power bounded with constant $K$. By virtue of the uniform Jordan property of $\A$, the quantity
\[
\theta=\sup_{T\in \A}\delta(T)^{-1}
\]
is finite. Hence, by Lemma \ref{L:powerbounded} we find
\[
\sup_{T\in \A}\max_{\lambda\in \Delta(T)}\sup_{p\in \bN} \{p|\lambda|^{p-1}\}\leq  \sup_{T\in \A}K\delta(T)^{-1}=K\theta.
\]
Choose $N\in \bN$ such that $N\geq 4\theta K$, along with $\eps>0$ so small that 
\[
(1-\eps)^{N-1}\geq 1/2.
\]
Assume now towards a contradiction that
\[
\sup_{T\in \A}\max_{\lambda\in \Delta(T)}|\lambda|\geq 1.
\]
We may thus find $S\in \A$ and $\mu\in \Delta(S)$ with the property that $|\mu|\geq 1-\eps$. Thus,
\[
N|\mu|^{N-1}\geq N(1-\eps)^{N-1}\geq N/2\geq 2\theta K
\]
which is absurd.
\end{proof}

Finally, we obtain a generalization of Theorem \ref{T:jointsimfinitepoly} which holds for possibly infinite families.

\begin{theorem}\label{T:jointsiminfinite}
Let $\A\subset \bM_n$ be a uniformly power bounded commuting family of matrices  with the uniform Jordan property. Then, there exists an invertible matrix $Y\in \bM_n$ with the property that $YT Y^{-1}$ is a contraction for every $T\in \A$. 
\end{theorem}
\begin{proof}
Choose $K>0$ such that every $T\in \A$ is power bounded with constant $K$. We note that by Lemma \ref{L:unifDelta} there is $0<r<1$ with the property that
\[
\max_{\lambda\in \Delta(T)}|\lambda|\leq r
\]
for every $T\in \A$. By Theorem \ref{T:jointsimfinitepoly}, there is a constant $\alpha\geq 1$ depending only on $\A$ such that for every finite subset $\F\subset \A$,  there is an invertible matrix $Y_\F\in \bM_n$ with the property that $Y_\F TY_\F^{-1}$ is a contraction for every $T\in\F$, and
\[
\|Y_\F\|=\|Y_\F^{-1}\|\leq \alpha \left( \frac{n^2 K}{1-r}\right)^{\frac{n-1}{2}}.
\]
Since closed balls of $\bM_n$ are compact in the norm topology, there is a subnet $(Y_{\F_\beta})_{\beta \in B}$ of $(Y_\F)_{\F\subset \A}$ which converges in norm to some invertible matrix $Y\in \bM_n$ with 
\[
\|Y\|=\|Y^{-1}\|\leq \alpha \left( \frac{n^2 K}{1-r}\right)^{\frac{n-1}{2}}.
\]
Let $T\in \A$. Then, there is $\beta_0\in B$ such that $T\in \F_\beta$ for every $\beta\geq \beta_0$. We conclude that $\|Y_{\F_\beta}TY_{\F_\beta}^{-1}\|\leq 1$ for $\beta\geq \beta_0$, whence
\[
\|Y TY^{-1}\|=\lim_{\beta\in B}\|Y_{\F_\beta}TY^{-1}_{\F_\beta}\|\leq 1
\]
and the proof is complete.
\end{proof}

At the time of this writing, it is unclear to us whether or not the uniform Jordan property may be removed from the assumptions of Theorem \ref{T:jointsiminfinite}. 

\bibliography{/Users/Raphael/Dropbox/Research/Bibliography/biblio_main}

\def\cprime{$'$}
\begin{bibdiv}
\begin{biblist}

\bib{constantin2010}{article}{
      author={Constantin, Olivia},
       title={A joint similarity problem for {$n$}-tuples of operators on
  vector-valued {B}ergman spaces},
        date={2010},
        ISSN={0022-1236},
     journal={J. Funct. Anal.},
      volume={258},
      number={8},
       pages={2682\ndash 2694},
         url={https://doi-org.uml.idm.oclc.org/10.1016/j.jfa.2009.10.012},
      review={\MR{2593338}},
}

\bib{CJ2009}{article}{
      author={Constantin, Olivia},
      author={Ja\"eck, Fr\'ed\'eric},
       title={A joint similarity problem on vector-valued {B}ergman spaces},
        date={2009},
        ISSN={0022-1236},
     journal={J. Funct. Anal.},
      volume={256},
      number={9},
       pages={2768\ndash 2779},
         url={https://doi-org.uml.idm.oclc.org/10.1016/j.jfa.2008.11.022},
      review={\MR{2502422}},
}

\bib{davidsonpaulsen1997}{article}{
      author={Davidson, Kenneth~R.},
      author={Paulsen, Vern~I.},
       title={Polynomially bounded operators},
        date={1997},
        ISSN={0075-4102},
     journal={J. Reine Angew. Math.},
      volume={487},
       pages={153\ndash 170},
      review={\MR{1454263}},
}

\bib{FG2002}{article}{
      author={Ferguson, Sarah~H.},
      author={Petrovic, Srdjan},
       title={The joint similarity problem for weighted {B}ergman shifts},
        date={2002},
        ISSN={0013-0915},
     journal={Proc. Edinb. Math. Soc. (2)},
      volume={45},
      number={1},
       pages={117\ndash 139},
         url={https://doi-org.uml.idm.oclc.org/10.1017/S0013091500000407},
      review={\MR{1884606}},
}

\bib{foguel1964}{article}{
      author={Foguel, S.~R.},
       title={A counterexample to a problem of {S}z.-{N}agy},
        date={1964},
        ISSN={0002-9939},
     journal={Proc. Amer. Math. Soc.},
      volume={15},
       pages={788\ndash 790},
      review={\MR{0165362 (29 \#2646)}},
}

\bib{halmos1970}{article}{
      author={Halmos, P.~R.},
       title={Ten problems in {H}ilbert space},
        date={1970},
        ISSN={0002-9904},
     journal={Bull. Amer. Math. Soc.},
      volume={76},
       pages={887\ndash 933},
  url={https://doi-org.uml.idm.oclc.org/10.1090/S0002-9904-1970-12502-2},
      review={\MR{0270173}},
}

\bib{HJ2013}{book}{
      author={Horn, Roger~A.},
      author={Johnson, Charles~R.},
       title={Matrix analysis},
     edition={Second},
   publisher={Cambridge University Press, Cambridge},
        date={2013},
        ISBN={978-0-521-54823-6},
      review={\MR{2978290}},
}

\bib{paulsen2002}{book}{
      author={Paulsen, Vern},
       title={Completely bounded maps and operator algebras},
      series={Cambridge Studies in Advanced Mathematics},
   publisher={Cambridge University Press},
     address={Cambridge},
        date={2002},
      volume={78},
        ISBN={0-521-81669-6},
      review={\MR{1976867 (2004c:46118)}},
}

\bib{paulsen1984}{article}{
      author={Paulsen, Vern~I.},
       title={Every completely polynomially bounded operator is similar to a
  contraction},
        date={1984},
        ISSN={0022-1236},
     journal={J. Funct. Anal.},
      volume={55},
      number={1},
       pages={1\ndash 17},
         url={http://dx.doi.org/10.1016/0022-1236(84)90014-4},
      review={\MR{733029 (86c:47021)}},
}

\bib{petrovic1997}{article}{
      author={Petrovi\'c, Srdjan},
       title={Polynomially unbounded product of two polynomially bounded
  operators},
        date={1997},
        ISSN={0378-620X},
     journal={Integral Equations Operator Theory},
      volume={27},
      number={4},
       pages={473\ndash 477},
         url={https://doi-org.uml.idm.oclc.org/10.1007/BF01192126},
      review={\MR{1442130}},
}

\bib{pisier1997}{article}{
      author={Pisier, Gilles},
       title={A polynomially bounded operator on {H}ilbert space which is not
  similar to a contraction},
        date={1997},
        ISSN={0894-0347},
     journal={J. Amer. Math. Soc.},
      volume={10},
      number={2},
       pages={351\ndash 369},
         url={http://dx.doi.org/10.1090/S0894-0347-97-00227-0},
      review={\MR{1415321 (97f:47002)}},
}

\bib{pisier1998}{article}{
      author={Pisier, Gilles},
       title={Joint similarity problems and the generation of operator algebras
  with bounded length},
        date={1998},
     journal={Integr. Equ. Oper. Theory},
      volume={31},
       pages={353\ndash 370},
}

\bib{pisier2003}{book}{
      author={Pisier, Gilles},
       title={Introduction to operator space theory},
      series={London Mathematical Society Lecture Note Series},
   publisher={Cambridge University Press, Cambridge},
        date={2003},
      volume={294},
        ISBN={0-521-81165-1},
         url={https://doi-org.uml.idm.oclc.org/10.1017/CBO9781107360235},
      review={\MR{2006539}},
}

\bib{pisier2007}{article}{
      author={Pisier, Gilles},
       title={Simultaneous similarity, bounded generation and amenability},
        date={2007},
        ISSN={0040-8735},
     journal={Tohoku Math. J. (2)},
      volume={59},
      number={1},
       pages={79\ndash 99},
  url={http://projecteuclid.org.proxy.lib.uwaterloo.ca/euclid.tmj/1176734749},
      review={\MR{2321994 (2008d:46075)}},
}

\bib{popescu2014sim}{article}{
      author={Popescu, Gelu},
       title={Similarity problems in noncommutative polydomains},
        date={2014},
        ISSN={0022-1236},
     journal={J. Funct. Anal.},
      volume={267},
      number={11},
       pages={4446\ndash 4498},
         url={https://doi-org.uml.idm.oclc.org/10.1016/j.jfa.2014.09.023},
      review={\MR{3269883}},
}

\bib{rota1960}{article}{
      author={Rota, Gian-Carlo},
       title={On models for linear operators},
        date={1960},
        ISSN={0010-3640},
     journal={Comm. Pure Appl. Math.},
      volume={13},
       pages={469\ndash 472},
         url={https://doi-org.uml.idm.oclc.org/10.1002/cpa.3160130309},
      review={\MR{0112040}},
}

\bib{vN1951}{article}{
      author={von Neumann, Johann},
       title={Eine {S}pektraltheorie f\"ur allgemeine {O}peratoren eines
  unit\"aren {R}aumes},
        date={1951},
        ISSN={0025-584X},
     journal={Math. Nachr.},
      volume={4},
       pages={258\ndash 281},
         url={https://doi-org.uml.idm.oclc.org/10.1002/mana.3210040124},
      review={\MR{0043386}},
}

\end{biblist}
\end{bibdiv}
\bibliographystyle{plain}


\end{document}